\documentclass[11pt]{amsart}

\usepackage{amsmath}
\usepackage{amssymb}
\usepackage{amsthm}
\usepackage{epsfig}  
\usepackage{mathrsfs}
\usepackage{color}	 %
\input{xy}
\xyoption{all}

\newtheorem{thm}{Theorem}[section]
\newtheorem{prop}[thm]{Proposition}
\newtheorem{lem}[thm]{Lemma}
\newtheorem{cor}[thm]{Corollary}

\theoremstyle{defn}
\newtheorem{definition}[thm]{Definition}
\newtheorem{example}[thm]{Example}

\theoremstyle{remark}

\newtheorem{remark}[thm]{Remark}

\numberwithin{equation}{section}

\DeclareMathSizes{5}{3}{2}{2}

\newcommand{\za}{\alpha}
\newcommand{\zb}{\beta}
\newcommand{\zd}{\delta}

\newcommand{\ze}{\epsilon}
\newcommand{\zg}{\gamma}
\newcommand{\zG}{\Gamma}
\newcommand{\zl}{\lambda}
\newcommand{\zs}{\sigma}
\newcommand{\zS}{\Sigma}

\newcommand{\torT}{\mathscr{T}(T)}
\newcommand{\ftorT}{\mathscr{F}(T)}

\newcommand{\calc}{\mathcal{C}}
\newcommand{\cald}{\mathcal{D}}
\newcommand{\cale}{\mathcal{E}}
\newcommand{\calj}{\mathcal{J}}
\newcommand{\cals}{\mathcal{S}}

\newcommand{\calt}{\mathcal{T}}

\newcommand{\calh}{\mathcal{H}}

\newcommand{\ZZ}{\mathbb{Z}}

\newcommand{\Hom}{\textup{Hom}}
\newcommand{\End}{\textup{End}}
\newcommand{\Ext}{\textup{Ext}}
\newcommand{\Ann}{\textup{Ann}}
\newcommand{\add}{\textup{add}}

\renewcommand{\DH}{\mathcal{D}^b(\mathcal{H})}
\newcommand{\DA}{\mathcal{D}^b(\textup{mod}\,A)}
\newcommand{\DC}{\mathcal{D}^b(\textup{mod}\,C)}

\newcommand{\E}{\Ext^2_C(DC,C)}
\newcommand{\CE}{C\ltimes \E}
\newcommand{\Ctilde}{\widetilde{C}}

\begin{document}
\title{Cluster-tilted and quasi-tilted algebras}
\author{Ibrahim Assem}
\address{D\'epartement de Math\'ematiques,
Universit\'e de Sherbrooke,
Sherbrooke, Qu\'ebec,
Canada J1K 2R1}
\email{ibrahim.assem@usherbrooke.ca}

\author{Ralf Schiffler}
\address{Department of Mathematics, University of Connecticut, 
Storrs, CT 06269-3009, USA}
\email{schiffler@math.uconn.edu}

\author{Khrystyna Serhiyenko}\thanks{The first author gratefully acknowledges partial support from the NSERC of Canada. The second author was supported by the NSF CAREER grant DMS-1254567 and by the University of Connecticut. The third author was  supported by the NSF Postdoctoral fellowship MSPRF-1502881.}
\address{Department of Mathematics, University of California, Berkeley, 
CA 94720-3840, USA}
\email{khrystyna.serhiyenko@berkeley.edu}

\begin{abstract}
In this paper, we prove that relation-extensions of quasi-tilted algebras are 2-Calabi-Yau tilted. With the objective of describing the module category of a cluster-tilted algebra of euclidean type, we define the notion of reflection so that any two local slices can be reached one from the other by a sequence of reflections and coreflections. We then give an algorithmic procedure for constructing the tubes of a cluster-tilted algebra of euclidean type. Our main result characterizes quasi-tilted algebras whose relation-extensions are cluster-tilted of euclidean type.
\end{abstract}

 \maketitle


\section{Introduction}
Cluster-tilted algebras were introduced by Buan, Marsh and Reiten \cite{BMR} and, independently in \cite{CCS} for type $\mathbb{A}$ as a byproduct of the now extensive theory of cluster algebras of Fomin and Zelevinsky \cite{FZ}. Since then, cluster-tilted algebras have been the subject of several investigations, see, for instance,
\cite{ABCP,ABS,BFPPT,BT, BOW,BMR2, KR,OS, SS,SS2}. 

In particular, in \cite{ABS} is given a construction procedure for cluster-tilted algebras: let $C$ be a triangular algebra of global dimension two over an algebraically closed field $k$, and consider the $C$-$C$-bimodule $\E$, where $D=\Hom_k(-,k)$ is the standard duality, with its natural left and right $C$-actions.
The trivial extension of $C$ by this bimodule is called the {\em relation-extension} $\Ctilde$ of $C$. It is shown there that, if $C$ is tilted, then its relation-extension is cluster-tilted, and every cluster-tilted algebra occurs in this way.

Our purpose in this paper is to study the relation-extensions of a wider class of triangular algebras of global dimension two, namely the class of quasi-tilted algebras, introduced by Happel, Reiten and Smal\o \  in \cite{HRS}. 
In general, the relation-extension of a quasi-tilted algebra is not cluster-tilted, however it is 2-Calabi-Yau tilted, see Theorem \ref{thm 2.1} below. We then look more closely at those cluster-tilted algebras which are tame and representation-infinite. According to \cite{BMR}, these coincide exactly with the cluster-tilted algebras of euclidean type. We ask then the following question: Given a cluster-tilted algebra $B$ of euclidean type, find all quasi-tilted algebras $C$ such that $B=\Ctilde$. A similar question has been asked (and answered) in \cite{ABS2}, where, however, $C$ was assumed to be tilted.

For this purpose, we generalize the notion of reflections of \cite{ABS4}. We prove that this operation allows to produce all tilted algebras $C$ such that $B=\Ctilde$, see Theorem \ref{thm mainreflection}. In \cite{ABS4} this result was shown only for cluster-tilted algebras of tree type.
We also prove that, unlike those of \cite{ABS4}, reflections in the sense of the present paper are always defined, that the reflection of a tilted algebra is also tilted of the same type, and that they have the same relation-extension, see Theorem \ref{thm reflection}  and Proposition \ref{prop reflection} below. Because all tilted algebras having a given cluster-tilted algebra as relation-extension are given by iterated reflections, this gives an algorithmic answer to our question above.

After that, we look at the tubes of a cluster-tilted algebra of euclidean type and give a procedure for constructing those tubes which contain a projective, see Proposition \ref{prop5}.

We then return to quasi-tilted algebras in our last section, namely we define a particular two-sided ideal of a cluster-tilted algebra, which we call the partition ideal. Our first result (Theorem \ref{thm ctaqt}) shows that the quasi-tilted algebras which are not tilted but have a given cluster-tilted algebra $B$ of euclidean type as relation-extension are the quotients of $B$ by a partition ideal.
We end the paper with the proof of our main result (Theorem \ref{thm ctaqt3}) which says that if $C$ is quasi-tilted and such that $B=\Ctilde$, then either $C$ is the quotient of $B$ by the annihilator of a local slice (and then $C$ is tilted) or it is the quotient of $B$ by a partition ideal (and then $C$ is not tilted except in two cases easy to characterize).

\section{Preliminaries}

\subsection{Notation}
Throughout this paper, algebras are basic and connected finite dimensional algebras over a fixed algebraically closed field $k$.  For an algebra $C$, we denote by $\text{mod}\,C$ the category of finitely generated right $C$-modules.  All subcategories are full, and identified with their object classes.  Given a category $\mathcal{C}$, we sometimes write $M\in\mathcal{C}$ to express that $M$ is an object in $\mathcal{C}$.  If $\mathcal{C}$ is a full subcategory of $\text{mod}\,C$, we denote by $\text{add}\,\mathcal{C}$ the full subcategory of $\text{mod}\,C$ having as objects the finite direct sums of summands of modules in $\mathcal{C}$.  

For a point $x$ in the ordinary quiver of a given algebra $C$, we denote by $P(x)$, $I(x)$, $S(x)$ respectively, the indecomposable projective, injective and simple $C$-modules corresponding to $x$.  We denote by $\Gamma(\text{mod}\,C)$ the Auslander-Reiten quiver of $C$ and by $\tau = D\text{Tr}, \tau^{-1} = \text{Tr} D$ the Auslander-Reiten translations.  For further definitions and facts, we refer the reader to \cite{ARS, ASS, S}.  

\subsection{Tilting}
Let $Q$ be a finite connected and acyclic quiver.  A module $T$ over the path algebra $kQ$ of $Q$ is called \emph{tilting} if $\text{Ext}^1_{kQ}(T,T)=0$ and the number of isoclasses (isomorphism classes) of indecomposable summands of $T$ equals $|Q_0|$, see \cite{ASS}.  An algebra $C$ is called \emph{tilted of type $Q$} if there exists a tilting $kQ$-module $T$ such that $C=\text{End}_{kQ} T$.  It is shown in \cite{Ri} that an algebra $C$ is tilted if and only if it contains a \emph{complete slice} $\Sigma$, that is, a finite set of indecomposable modules such that 
\begin{itemize}
\item[1)] $\bigoplus_{U\in \Sigma} U$ is a sincere $C$-module.  
\item[2)] If $U_0\to U_1 \to \dots \to U_t$ is a sequence of nonzero morphisms between indecomposable modules with $U_0,U_t\in\Sigma$ then $U_i\in\Sigma$ for all $i$ (\emph{convexity}).
\item[3)] If $0\to L \to M \to N \to 0$ is an almost split sequence in $\text{mod}\,C$ and at least one indecomposable summand of $M$ lies in $\Sigma$, then exactly one of $L,N$ belongs to $\Sigma$.  
\end{itemize}

For more on tilting and tilted algebras, we refer the reader to \cite{ASS}.  

Tilting can also be done within the framework of a hereditary category.  Let $\mathcal{H}$ be an abelian $k$-category which is Hom-finite, that is, such that, for all $X,Y\in\mathcal{H}$, the vector space $\text{Hom}_{\mathcal{H}}(X,Y)$ is finite dimensional.  We say that $\mathcal{H}$ is \emph{hereditary} if  $\text{Ext}^2_{\mathcal{H}}(-, ?)=0$.  An object $T\in\mathcal{H}$ is called a \emph{tilting object} if $\text{Ext}^1_{\mathcal{H}}(T,T)=0$ and the number of isoclasses of indecomposable objects of $T$ is  the rank of the Grothendieck group $K_0(\mathcal{H})$.  

The endomorphism algebras of tilting objects in hereditary categories are called \emph{quasi-tilted algebras}.  For instance, tilted algebras but also canonical algebras (see \cite{Ri}) are quasi-tilted.  Quasi-tilted algebras have attracted a lot of attention and played an important role in representation theory, see for instance \cite{HRS, Sk}.  

\subsection{Cluster-tilted algebras} 
Let $Q$ be a finite, connected and acyclic quiver.  The \emph{cluster category} $\mathcal{C}_Q$ of $Q$ is defined as follows, see \cite{BMRRT}.  Let $F$ denote the composition $\tau^{-1}_{\mathcal{D}}[1]$, where $\tau^{-1}_{\mathcal{D}}$ denotes the inverse Auslander-Reiten translation in the bounded derived category $\mathcal{D} = \mathcal{D}^b(\text{mod}\, kQ)$, and [1] denotes the shift of $\mathcal{D}$.  Then $\mathcal{C}_Q$ is the orbit category $\mathcal{D}/F$: its objects are the $F$-orbits $\widetilde{X}=(F^i X)_{i\in\mathbb{Z}}$ of the objects $X\in\mathcal{D}$, and the space of morphisms from $\widetilde{X}=(F^i X)_{i\in\mathbb{Z}}$ to $\widetilde{Y}=(F^i Y)_{i\in\mathbb{Z}}$ is 
$$\text{Hom}_{\mathcal{C}_Q}(\widetilde{X}, \widetilde{Y}) = \bigoplus_{i\in\mathbb{Z}} \text{Hom}_{\mathcal{D}}(X, F^i Y).$$
Then $\mathcal{C}_Q$ is a triangulated category with almost split triangles and, moreover, for $\widetilde{X}, \widetilde{Y}\in\mathcal{C}_Q$ we have a bifunctorial isomorphism $\text{Ext}^1_{\mathcal{C}_Q}(\widetilde{X}, \widetilde{Y})\cong D\text{Ext}^1_{\mathcal{C}_Q}(\widetilde{Y},\widetilde{X})$.  This is expressed by saying that the category $\mathcal{C}_Q$ is \emph{2-Calabi-Yau}.

An object $\widetilde{T}\in\mathcal{C}_Q$ is called \emph{tilting} if $\text{Ext}^1_{\mathcal{C}_Q}(\widetilde{T}, \widetilde{T})=0$ and the number of isoclasses of indecomposable summands of $\widetilde{T}$ equals $|Q_0|$.  The endomorphism algebra $B=\text{End}_{\mathcal{C}_Q} \widetilde{T}$ is then called \emph{cluster-tilted} of type $Q$.  
More generally, the endomorphism algebra $\End_\calc \widetilde{T}$ of a tilting object $\widetilde{T}$ in  a $2$-Calabi-Yau category with finite dimensional Hom-spaces is called a {\em 2-Calabi-Yau tilted algebra}, see \cite{Reiten}.

Let now $T$ be a tilting $kQ$-module, and $C=\text{End}_{kQ} T$ the corresponding tilted algebra.  Then it is shown in \cite{ABS} that the trivial extension $\widetilde{C}$ of $C$ by the $C$-$C$-bimodule $\text{Ext}^2_C (DC,C)$ with the two natural actions of $C$, the so-called \emph{relation-extension} of $C$, is cluster-tilted.  Conversely, if $B$ is cluster-tilted, then there exists a tilted algebra $C$ such that $B=\widetilde{C}$.  

Let now $B$ be a cluster-tilted algebra, then a full subquiver $\Sigma$ of $\Gamma(\text{mod}\,B)$ is a \emph{local slice}, see \cite{ABS2}, if: 
\begin{itemize}
\item[1)] $\Sigma$ is a \emph{presection}, that is, if $X\to Y$ is an arrow then: 
\begin{itemize}
\item[(a)] $X\in\Sigma$ implies that either $Y\in\Sigma$ or $\tau Y \in \Sigma$
\item[(b)] $Y\in\Sigma$ implies that either $X\in \Sigma$ or $\tau^{-1} X\in\Sigma$.
\end{itemize}
\item[2)] $\Sigma$ is \emph{sectionally convex}, that is, if $X=X_0\to X \to \dots \to X_t = Y$ is a sectional path in $\Gamma(\text{mod}\,B)$ then $X,Y\in\Sigma$ implies that $X_i\in\Sigma$ for all $i$.  
\item[3)] $|\Sigma_0| = \text{rk}\,K_0(B)$.  
\end{itemize}

Let $C$ be tilted, then, under the standard embedding $\text{mod}\,C \to \text{mod}\,\widetilde{C}$, any complete slice in the tilted algebra $C$ embeds as a local slice in $\text{mod}\,\widetilde{C}$, and any local slice in $\textup{mod}\,\Ctilde$ occurs in this way.  If $B$ is a cluster-tilted algebra, then a tilted algebra $C$ is such that $B=\widetilde{C}$ if and only if there exists a local slice $\Sigma$ in $\Gamma(\text{mod}\,B)$ such that $C=B/\text{Ann}_B \Sigma$, where $\text{Ann}_B \Sigma = \bigcap_{X\in\Sigma} \text{Ann}_B X$, see \cite{ABS2}.

Let $\Sigma$ be a local slice in the transjective component of $\Gamma(\text{mod}\,B)$ having the property that all the sources in $\Sigma$ are injective $B$-modules.  Then $\Sigma$ is called a \emph{rightmost} slice of $B$.  Let $x$ be a point in the quiver of $B$ such that $I(x)$ is an injective source of the rightmost slice $\Sigma$.  Then $x$ is called a \emph{strong sink}.  \emph{Leftmost slices} and \emph{strong sources} are defined dually.

\section{From quasi-tilted to cluster-tilted algebras}

We start with a motivating example. Let $C$ be the tilted algebra of type $\widetilde{\mathbb{A}}$ given by the quiver
\[\xymatrix@R5pt@C60pt {&2\ar[ld]_\zb\\1&&4\ar[lu]_\za\ar[ld]^\zg\\&3\ar[lu]^\zd}\]
bound by $\za\zb=0$, $\zg\zd=0$. Its relation-extension is the cluster-tilted algebra $B$ given by the quiver 
\[\xymatrix@R25pt@C60pt {&2\ar[ld]_\zb\\1\ar@<1.5pt>[rr]^\zl\ar@<-1.5pt>[rr]_\mu&&4\ar[lu]_\za\ar[ld]^\zg\\&3\ar[lu]^\zd}\]
bound by $\za\zb=0$, $\zb\zl=0$, $\zl\za=0$, $\zg\zd=0$, $\zd\mu=0$, $\mu\zg=0$. 
However, $B$ is also the relation-extension of the algebra $C'$ given by the quiver
\[\xymatrix@R=30pt@C60pt {2&4\ar[l]_\za&1\ar@<1.5pt>[l]_{ ^\zl }\ar@<-1.5pt>[l]^{\ \atop\mu}&3\ar[l]_\zd}\]
bound by $\zl\za=0$, $\zd\mu=0$.
This latter algebra $C'$ is not tilted, but it is quasi-tilted. In	 particular, it is triangular of global dimension two. Therefore, the question arises natrually whether the relation-extension of a quasi-tilted algebra is always cluster-tilted. This is certainly not true in general, for the relation-extension of a tubular algebra is not cluster-tilted. However, it is 2-Calabi-Yau tilted. In this section, we prove that the relation-extension
of  a quasi-tilted algebra is always 2-Calabi-Yau tilted.

Let $\calh$ be a hereditary category with tilting object $T$. Because of \cite{Happel}, there exist an algebra $A$, which is hereditary or canonical, and a triangle equivalence $\Phi:\DH\to\DA$.
Let $T'$ denote the image of $T$ under this equivalence. Because $\Phi$ preserves the shift and the Auslander-Reiten translation, it induces an  equivalence between the cluster categories $\calc_\calh$ and $\calc_A$, see \cite[Section 4.1]{Amiot}. Indeed, because $A$ is canonical or hereditary, it follows that $\calc_A\cong\DA/F$, where $F=\tau^{-1}[1].$
Therefore, we have $\End_{\calc_\calh}T\cong\End_{\calc_A} T'$.

We say that a 2-Calabi-Yau tilted algebra $\End_\calc T$ is of \emph{canonical type} if the 2-Calabi-Yau category $\calc$ is the cluster category of a canonical algebra.
The proof of the next theorem follows closely \cite{ABS}.

\begin{thm}
 \label{thm 2.1}
 Let $C$ be a quasi-tilted algebra. Then its relation-extension $\widetilde{C}$ is cluster-tilted or it is 2-Calabi-Yau titled of canonical type.
\end{thm}
\begin{proof}
 Because $C$ is quasi-tilted, there exist a hereditary category $\calh$ and a tilting object $T$ in $\calh$ such that $C=\End_\calh T$. As observed above, there exist an algebra $A$, which is hereditary or canonical, and a triangle equivalence $\Phi:\DH\to\DA$. Let $T'=\Phi(T)$.We have
 $\DC\cong\DA\cong\DH$, and therefore
 \[
\begin{array}
 {rcl}
 \E&\cong & \Hom_{\DC}(\tau C[1] , C[2]) \\
 &\cong & \Hom_{\DH}(\tau T[1] , T[2]) \\
 &\cong & \Hom_{\DH}( T , \tau^{-1} T[1]) \\
 &\cong & \Hom_{\DH}( T , F T) .\\
\end{array}
 \]
 Thus the additive structure of $\CE$ is that of 
  \[
\begin{array}
 {rcl}
 C\oplus \E&\cong & \End_{\calh}(T)\oplus \Hom_{\DH}(T,FT)\\
&\cong & \oplus_{i\in\ZZ}\Hom_{\DH}(T,FT)\\
&\cong & \Hom_{\calc_\calh}(T,T)\\
&\cong &\End_{\calc_\calh} T.
\end{array}
 \]
Then, we check exactly as in \cite[Section 3.3]{ABS} that the multiplicative structure is preserved. This completes the proof.
\end{proof}

Let $C$ be a representation-infinite quasi-tilted algebra. Then $C$ is derived equivalent to a hereditary or a  canonical algebra $A$. Let $n_A$ denote the tubular type of $A$. We then say that $C$ has canonical type $n_C=n_A$.

\begin{lem}\label{lem 1}
 Let $C$ be a representation-infinite  quasi-tilted. Then its relation-extension $\Ctilde$  is cluster-tilted of euclidean type if and only if $n_C$ is one of 
 \[(p,q),(2,2,r),(2,3,3),(2,3,4),(2,3,5), \textup{ with $p\le q$, $2\le r.$}\]
 \end{lem}
 \begin{proof}
 Indeed, $\Ctilde$ is cluster-tilted of euclidean type if and only if $C$ is derived equivalent to a tilted algebra of euclidean type, and this is the case if and only if $n_C$ belongs to the above list.
\end{proof}

\begin{remark}
 \label{rem 2} It is possible that $C$ is domestic, but yet $\Ctilde$ is wild. Indeed, we modify the example after Corollary D in \cite{Sk}. 
 Recall from \cite{Sk} that there exists a tame concealed full convex subcategory $K$ such that $C$ is a semiregular branch enlargement of $K$ 
 \[C=[E_i]K[F_j],\]
 where $E_i, F_j$ are (truncated) branches. Then the representation theory of $C$ is determined by those of $C^-=[E_i]K$ and $C^+=K[F_j]$. 
 Let $C$ be given by the quiver
 
\[\xymatrix@R20pt@C40pt{
1&&&&6\ar[dl]_\zd&11\ar[l]_\zeta\\
&3\ar[ul]_\za\ar[dl]^\zb&4\ar[l]_\zg\ar[d]^\nu&5\ar[l]_\zs \\
2&&8\ar[d]^\varphi &9\ar[l]_\omega&7\ar[ul]_\rho
\\&&10
}
\]
bound by the relations $\zs\nu=0$, $\omega\varphi=0$, $\zeta\zd\zs\zg\zb=0$. Here $C^-$ is the full subcategory generated by $C_0\setminus\{11\}$ and $C^+$ the one generated by $C_0\setminus\{8,9,10\}$.
Then $C^-$ has domestic tubular type $(2,2,7)$ and $C^+$ has domestic tubular type $(2,3,4)$. Therefore $C$ is domestic. On the other hand, the canonical type of $C$ is $(2,3,7)$, which is wild. In this example, the 2-Calabi-Yau tilted algebra $\Ctilde$ is not cluster-tilted, because it is not of euclidean type, but the derived category of $\textup{mod}\,C$ contains tubes, see \cite{Ringel Durham}.
\end{remark}

\begin{remark}
 There clearly exist algebras which are not quasi-tilted but whose relation-extension is cluster-tilted of euclidean type.  Indeed, let $C$ be given by the quiver
 \[\xymatrix@C40pt{6\ar[r]^\za&5\ar[r]^\zb&4\ar[r]^\zg&3\ar[r]^\zd&2\ar@<-2pt>[r]_\mu\ar@<2pt>[r]^\zl&1}
 \]
 bound by $\za\zb=0,\zd\zl=0$. Then $C$ is iterated tilted of type $\widetilde{\mathbb{A}}$ of global dimension 2, see \cite{FPT}. Its relation-extension is given by
 \[\xymatrix@C40pt{6\ar[r]^\za&5\ar[r]^\zb&4\ar@/_25pt/[ll]_\zs\ar[r]^\zg&3\ar[r]^\zd&2\ar@<-2pt>[r]_\mu\ar@<2pt>[r]^\zl&1\ar@/_25pt/[ll]_\eta}
 \]
 bound by $\za\zb=0,\zb\zs=0,\zs\za=0,\zd\zl=0,\zl\eta=0,\eta\zd=0$. 
This algebra is isomorphic to the relation-extension of the tilted algebra of type $\widetilde{\mathbb{A}}$ given by the quiver
\[\xymatrix@R20pt@C40pt{6\\
&4\ar[lu]^\zs \ar[r]^\zg&3\ar[r]^\zd&2\ar@<-2pt>[r]_\mu\ar@<2pt>[r]^\zl&1\\
5\ar[ru]^\zb}
\]
bound by $\zb\zs=0$, $\zd\zl=0$. Therefore $\Ctilde$ is cluster-tilted of euclidean type. On the other hand, $C$ is not quasi-tilted, because the uniserial module $\begin{smallmatrix}4\\3\end{smallmatrix}$ has both projective and injective dimension 2. 
\end{remark}


\section{Reflections}
Let $C$ be a tilted algebra. Let $\zS$ be a rightmost slice, and let $I(x)$ be an injective source of $\zS$. Thus $x$ is a strong sink in $C$. 
\begin{definition} We define
 the \emph{completion $H_x$ of $x$} by the following three conditions.
\begin{itemize}
\item [(a)] $I(x)\in H_x$.
\item [(b)] $H_x$ is closed under predecessors in $\zS$.
\item [(c)] If $L\to M$ is an arrow in $\zS$ with $L\in H_x$ having an injective successor in $H_x$ then $M\in H_x$.
\end{itemize}
\end{definition}

Observe that $H_x$ may be constructed inductively in the following way. We let $H_1=I(x)$, and $H_2'$ be the closure of $H_1$ with respect to (c) (that is, we simply add the direct successors of $I(x)$ in $\zS$, and 
if a direct successor of $I(x)$ is injective, we also take its direct successor, etc.) We then let $H_2$ be the closure of $H_2'$ with respect to predecessors in $\zS$. Then we repeat the procedure; given $H_i$, we let $H_{i+1}'$ be the closure of $H_i$ with respect to (c) and $H_{i+1}$ be the closure of $H_{i+1}'$ with respect to predecessors. This procedure must stabilize, because the slice $\zS$ is finite. If $H_j=H_k$ with $k>j$
, we let $H_x=H_j$.

We can decompose $H_x$ as the disjoint union of three sets as follows. Let $\calj$ denote the set of injectives in $H_x$, let $\calj^-$ be the set of non-injectives in $H_x$ which have an injective successor in $H_x$, and let $\cale=H_x\setminus(\calj\cup\calj^-)$ denote the complement of $(\calj\cup\calj^-)$ in $H_x$. Thus
\[H_x=\calj\sqcup\calj^-\sqcup\cale\]
is a disjoint union.
\begin{remark}\label{rem H}
If $\calj^-=\emptyset$ then $H_x$ reduces to the completion $G_x$ as defined in \cite{ABS4}. Recall that $G_x$ does not always exist, but, as seen above, $H_x$ does.
Conversely, if $G_x$ exists, then it follows from its construction in \cite{ABS4} that $\calj^-=\emptyset$. 

Thus $\calj^-=\emptyset$ if and only if $G_x$ exists, and, in this case $G_x=H_x$.
\end{remark}

For every module $M$ over a cluster-tilted algebra $B$, we can consider a lift $\widetilde M$ in the cluster category $\calc$. Abusing notation, we   sometimes write $\tau^i M$ to denote the image of $\tau^i_\calc \widetilde M$ in $\textup{mod}\,B$, and say that the Auslander-Reiten translation is computed in the cluster category.

\begin{definition}
 Let $x$ be a strong sink in $C$ and let $\zS$ be a rightmost local slice with injective source $I(x)$. Recall that $\zS$ is also a local slice in $\textup{mod}\,B$. Then the reflection of the slice $\zS$ in $x$ is 
 \[\zs_x^+\zS=\tau^{-2}(\calj\cup\calj^-)\cup\tau^{-1}\cale\cup(\zS\setminus H_x),\]
 where $\tau$ is computed in the cluster category. 
 In a similar way, one defines the coreflection $\zs^-_y$ of leftmost slices with projective sink $P_C(y)$.
\end{definition}

\begin{thm}\label{thm reflection}
 Let $x$ be a strong sink in $C$ and let $\zS$ be a rightmost local slice in $\textup{mod}\,B$  with injective source $I(x)$. Then the reflection $\zs_x^+\zS$ is a local slice as well.
\end{thm}
\begin{proof}
 Set $\zS'=\zs_x^+\zS$ and 
 \[\zS''=\tau^{-1}(\calj\cup\calj^-)\cup\tau^{-1}\cale\cup(\zS\setminus H_x)=\tau^{-1}H_x\cup(\zS\setminus H_x),\]
where again, $\zS''$ and $\tau$ are computed in the cluster category $\calc$. We claim that  $\zS''$ is a local slice in $\calc$. Notice  that since $H_x$ is closed under predecessors in $\zS$, then, if $X\in\zS\setminus H_x$ is a neighbor of $Y\in H_x$, we must have an arrow $Y\to X$ in $\zS$. This observation being made, $\zS''$ is clearly obtained from $\zS$ by applying a sequence of APR-tilts. Thus $\zS''$ is a local slice in $\calc$.

We now claim that $\tau^{-1}(\calj\cup\calj^-)$ is closed under predecessors in $\zS''$. Indeed, let $X\in\tau^{-1}(\calj\cup\calj^-)$ and $Y\in \zS''$ be such that we have an arrow $Y\to X$. Then, there exists an arrow $\tau X\to Y$ in the cluster category. Because $X\in \tau^{-1}(\calj\cup\calj^-)$, we have $\tau X\in \calj\cup\calj^-$.
Now if $Y\in \zS$, then the arrow $\tau X\to Y$ would imply that $Y\in H_x$, which is impossible, because $Y\in \zS''$ and $\zS''\cap H_x=\emptyset$.
 Thus $Y\notin \zS$, and therefore $Y\in( \zS''\setminus\zS)=\tau^{-1}H_x$.
Hence $\tau Y\in H_x$.
Moreover, there is an arrow $\tau Y\to \tau X$. Using that $\tau X\in \calj\cup\calj^-$, this implies that  $\tau Y$ has an injective successor in $H_x$ and thus $Y\in \tau^{-1}(\calj\cup\calj^-)$. This establishes our claim that $\tau^{-1}(\calj\cup\calj^-)$ is closed under predecessors in $\zS''$.

Thus applying the same reasoning as before, we get that
\[\zS'=(\zS''\setminus \tau^{-1}(\calj\cup\calj^-))\cup\tau^{-2}(\calj\cup\calj^-)
\]
is a local slice in $\calc$. Now  we claim that \[\zS'\cap \add(\tau T)=\emptyset.\]
First, because $\zS\cap \add(\tau T)=\emptyset$, we have $(\zS\setminus H_x)\cap \add(\tau T) =\emptyset$.
Next, $\cale$ contains no injectives, by definition. Thus $\tau^{-1}\cale\cap \add(\tau T)=\emptyset.$ Assume now that $X\in \add(\tau T)$ belongs to $\tau^{-2}\calj^-$. Then $\tau^2 X\in H_x$ and there exists an injective predecessor $I(j)$ of $\tau^2 X$ in $H_x$, and since $H_x$ is part of the local slice $\zS$, there exists a sectional path from $I(j)$ to $\tau^2 X$. Applying $\tau^{-2}$, we get a sectional path from $T_j$  to $X$ in the cluster category. But this means $\Hom_\calc(T_j,X)\ne 0$, which is a contradiction to the hypothesis that $X\in \add(\tau T)$.
Finally, if $X\in \tau^{-2}\calj$ then $X$ is a summand of $T$, which, again, is contradicting the hypothesis that $X\in \add(\tau T)$.
\end{proof}

\smallskip
Following \cite{ABS4}, let $\cals_x$ be the full subcategory of $C$ consisting of those $y$ such that $I(y)\in H_x$.
\begin{lem}
 \label{lem S}
 \begin{itemize}
\item [\textup{(a)}] $\cals_x$ is hereditary.
\item [\textup{(b)}] $\cals_x$ is closed under successors in $C$.
\item [\textup{(c)}] $C$ can be written in the form
\[ C= \left[
\begin{array}
 {cc} H&0\\M&C'
\end{array}\right] ,\]
where $H$ is hereditary, $C'$ is tilted and $M$ is a $C'$-$H$-bimodule.
\end{itemize}
\end{lem}
\begin{proof}
 (a) Let $H=\End(\oplus_{y\in\cals_x}I(y)).$ Then $H$ is a full subcategory of the hereditary endomorphism algebra of $\zS$. Therefore $H$ is also hereditary, and so $\cals_x$ is hereditary.
 
 (b) Let $y\in \cals_x$ and $y\to z$ in $C$. Then there exists  a morphism $I(z)\to I(y)$. Because $I(z)$ is an injective $C$-module and $\zS$ is sincere, there exist a module $N\in \zS$ and a non-zero morphism $N\to I(z)$. Then we have a path $N\to I(z)\to I(y)$, and since $N,I(y)\in \zS$, we get that $I(z)\in \zS$ by convexity of the slice $\zS$ in $\textup{mod}\,C$. Moreover, since $I(y)\in H_x$ and $H_x$ is closed under predecessors in $\zS$, it follows that $I(z)\in H_x$. Thus $z\in \cals_x$ and this shows (b).
 
 (c) This follows from (a) and (b).
\end{proof}
We recall that the cluster duplicated algebra was introduced in \cite{ABS3}.
\begin{cor}
 \label{cor dup}
 The cluster duplicated algebra $\overline{C}$ of $C$ is of the form
 \[ \overline{C}= \left[
\begin{array}
 {cccc} 
H&0&0&0\\
M&C'&0&0\\
0&E_0&H&0\\
0&E_1&M&C'
\end{array}
 \right]
 \]
 where $E_0=\Ext^2_C(DC',H)$ and  $E_1=\Ext^2_C(DC',C')$.
\end{cor}
\begin{proof}
 
We start by writing $C$ in the matrix form of the lemma. By definition, $H$ consists of those $y\in C_0$ such that the corresponding injective $I(y)$ lies in $ H_x$ inside the slice $\zS$. In particular, the projective dimension of these injectives is at most 1, hence $\E=\Ext^2_C(DC',C)$. The result now follows upon multiplying by idempotents.
\end{proof}

\begin{definition}
 Let $x$ be a strong sink in $C$. The reflection at $x$ of the algebra $C$ is 
 \[\zs_x^+ C= \left[
\begin{array}
 {cc}C'&0\\E_0&H
\end{array}\right]\]
where $E_0=\Ext^2_C(DC',H)$.
\end{definition}

\begin{prop}\label{prop reflection}
 The reflection $\zs_x^+C$ of $C$ is a tilted algebra having $\zs_x^+\zS$ as a complete slice. Moreover the relation-extensions of  $C$ and  $\zs_x^+\zS$ are isomorphic.
\end{prop}
\begin{proof}
 We first claim that the support $\textup{supp}(\zs_x^+\zS)$ of $\zs_x^+\zS$ is contained in $ \zs_x^+C$. 
 Let $X\in\zs_x^+\zS$. Recall that
 $\zs_x^+\zS=\tau^{-2}(\calj\cup\calj^-)\cup\tau^{-1}\cale\cup(\zS\setminus H_x)$. 
 If $X\in \tau^{-2}\calj$, then $X=P(y')$ is projective corresponding to a point $y'\in H$. Thus $I(y)\in H_x$ and the radical of $P(y)$ has no non-zero morphism into $I(y)$. Therefore $\textup{supp}(X)\subset \zs^+_XC$.
 
Assume next that $X\in \tau^{-2}\calj^-$, that is, $X=\tau^{-2}Y$, where $Y\in \calj^-$ has an injective successor $I(z)$ in $H_x$. Because all sources in $\zS$ are injective, there is an injective $I(y') \in \zS$ and a sectional path 
$I(y')\to\ldots\to Y\to \ldots \to I(z)$.
Applying $\tau^{-2}$, we obtain a sectional path
$P(y')\to\ldots\to X\to \ldots \to P(z)$.
In particular the point $y'$ belongs to the support of $X$.
Assume that there is a point $h$ in $H$ that is in the support of $X$. Then there exists a nonzero morphism $X\to I(h)$. But $I(h)\in \zS$ and there is no morphism from $X\in \tau^{-2}\zS$ to $\zS$. Therefore $\textup{supp}(X)\subset \zs^+_xC$.

By the same argument, we show that if $X\in \tau^{-1}\cale$, then  $\textup{supp}(X)\subset \zs^+_xC$.

Finally, all modules of $\zS\setminus H_x$ are supported in $C'$. This establishes our claim.

Now, by Theorem \ref{thm reflection},  $ \zs^+_x\zS$ is a local slice in  $\textup{mod}\, \Ctilde.$ Therefore $\Ctilde/\Ann\,\zs^+_x\zS$ is a tilted algebra in which $\zs^+_x\zS$ is a complete slice. Since the support of $\zs^+_x\zS$ is the same as the support of $\zs^+_x C$, we are done.
\end{proof}

\smallskip
We now come to the main result of this section, which states that any two tilted algebras that have the same relation-extension are linked to each other by a sequence of reflections and coreflections.

\begin{definition}
Let $B$ be a cluster-tilted algebra and let $\Sigma$ and $\Sigma'$ be two local slices in $\text{mod}\,B$.  We write $\Sigma \sim \Sigma'$ whenever $B/\Ann\,\Sigma = B/\Ann\,\Sigma'$.  
\end{definition}

\begin{lem}\label{lem 310}
 Let $B$ be a cluster-tilted algebra, and $\zS_1, \zS_2 $ be two local slices in $\textup{mod}\,B$.  Then there exists a sequence of reflections and coreflections $\sigma$ such that 
 \[ \sigma \Sigma_1\sim \Sigma_2.\]
\end{lem}

\begin{proof}
Given a local slice $\Sigma$ in $\text{mod}\,B$ such that $\Sigma$ has injective successors in the transjective component $\mathcal{T}$ of $\Gamma(\text{mod}\,B)$, let $\Sigma^+$ be the rightmost local slice such that $\Sigma\sim \Sigma^+$.  Then $\Sigma^+$ contains a strong sink $x$, thus reflecting in $x$ we obtain a local slice $\sigma^+_{x}\Sigma^+$ that has fewer injective successors in $\mathcal{T}$ than $\Sigma$.  To simplify the notation we define $\sigma^+_x\Sigma = \sigma^+_{x}\Sigma^+$. Similarly, we define $\sigma^-_y\Sigma=\sigma^-_y\Sigma^-$, where $\Sigma^-$ is the leftmost local slice containing a strong source $y$ and $\Sigma\sim\Sigma^-$.

Since we can always reflect in a strong sink, there exist sequences of reflections such that 
\[ \sigma^+_{x_r} \cdots \sigma^+_{x_2}\sigma^+_{x_1} \Sigma_1 = \Sigma^1_{\infty}\]
\[ \sigma^+_{y_s} \cdots \sigma^+_{y_2}\sigma^+_{y_1} \Sigma_2 = \Sigma^2_{\infty}\]
and $\Sigma^1_{\infty}, \Sigma^2_{\infty}$ have no injective successors in $\mathcal{T}$.  This implies that $\Sigma^1_{\infty}\sim\Sigma^2_{\infty}$.  Let 
\[\sigma = \sigma^-_{y_1} \sigma^-_{y_{2}}\cdots \sigma^-_{y_s}\sigma^+_{x_r} \cdots \sigma^+_{x_2}\sigma^+_{x_1}\]
thus $\sigma\Sigma_1\sim \Sigma_2$.   
\end{proof}
\begin{thm}\label{thm mainreflection}
 Let $C_1$ and $C_2$ be two tilted algebras that have the same relation-extension.   Then there exists a sequence of reflections and coreflections $\sigma$ such that $\sigma C_1 \cong C_2$.
\end{thm}
\begin{proof}
 Let $B$ be the common relation-extension of the tilted algebras $C_1$ and $C_2$. By \cite{ABS2}, there exist local slices $\zS_i$ in $\textup{mod}\,B$ such that $C_i=B/\Ann\,\zS_i$, for $i=1,2$.
 Now the result follows from  Lemma \ref{lem 310} and Theorem~\ref{thm reflection}.
\end{proof}

\begin{example}
 Let $A$ be the path algebra of the quiver
 \[\xymatrix@R2pt@C10pt{ & & 1\ar@/_10pt/[lldd]\ar[ld]\\
 & 2\ar[ld] \\
 3\\
 &4\ar[lu]\\
 &&5\ar@<1pt>[lu]\ar@<-1pt>[lu]\ar[ld]\\
 &6  
  }
 \]
 Mutating at the vertices 4,5, and 2 yields the cluster-tilted algebra $B$ with quiver
 
  \[\xymatrix@R2pt@C10pt{ & & 1\ar@<0pt>@/_10pt/[lldd]\ar@<-2pt>@/_10pt/[lldd]\\
 & 2\ar[ur] \\
 3\ar[ur]\ar@<0pt>@/^15pt/[rrdd]\ar@<-2pt>@/^15pt/[rrdd]\\
 &4\ar@<2pt>[lu]\ar@<-2pt>[lu]\ar[ul]\ar@<1pt>[dd]\ar@<-1pt>[dd]\\
 &&5\ar@<1pt>[lu]\ar@<-1pt>[lu]\\
 &6\ar[ur]
  }
 \]
 In the Auslander-Reiten quiver of $\textup{mod}\,B$ we have the following local configuration.
%

  \[\xymatrix@!@R0pt@C1pt{ 
  && I(1)\ar[rd]&& \circ&& P(1) \\
  & 1\ar[ru]&& 2\ar[rd] && {\begin{smallmatrix} 3\\ 5\,5\\4\end{smallmatrix}}\ar[ru] \\
  I(3)\ar[ru]\ar@/^15pt/[rruu] \ar[rd]&& \circ&& P(3)\ar[ru]\ar@/^15pt/[rruu]\\
&  {\begin{smallmatrix} 5555\\444 \end{smallmatrix}}\ar@<1pt>[rd]\ar@<-1pt>[rd] &&   {\begin{smallmatrix} 55\\4 \end{smallmatrix}}\ar@<1pt>[rd]\ar@<-1pt>[rd] \ar[ru]\\
  && {\begin{smallmatrix} 555\\44 \end{smallmatrix}}\ar@<1pt>[ru]\ar@<-1pt>[ru] && 5 \ar[rd] \\
  &I(6) \ar[ru] &&\circ && P(6)
  }
 \] \bigskip
where 
\[ \begin{array}{ccc} I(1)= {\begin{smallmatrix} 2\\1\end{smallmatrix}} & I(3)=  {\begin{smallmatrix} 2\ \ 5555\\11\ 444 \\ 3\end{smallmatrix}} &
I(6)= \begin{smallmatrix} 555\\44\\6\end{smallmatrix} \end{array}\]
\smallskip

The 6 modules on the left form a rightmost local slice $\zS$ in which both $I(3)$ and $I(6)$ are sources, so 3 and 6 are strong sinks.
For both strong sinks the subset $\calj^-$ of the completion consists of the simple module $1$. The simple module $2=\tau^{-1}1$ does not lie on a local slice.

  The completion $H_6$ is the whole local slice $\zS$ and therefore the reflection $\zs_6^+\zS$ is the local slice consisting of the 6 modules on the right containing both $P(1)$ and $P(6)$.

On the other hand, the completion $H_3$ consists of the four modules $I(3)$, $S(1)$, $I(1) $ and   
${\begin{smallmatrix} 5555\\444 \end{smallmatrix}}$,
and therefore
 the reflection $\zS'=\zs_3^+\zS$ is the local slice consisting of the 6 modules on the straight line from $I(6)$ to $P(1)$.
This local slice admits the strong sink $6$ and the completion $H'_6$ in $\zS'$ consists of the two modules $I(6)$ and  ${\begin{smallmatrix} 555\\44 \end{smallmatrix}}$. Therefore the reflection $\zs_6^+\zS'$ is equal to $\zs_6^+\zS$. 
Thus 
\[\zs_6^+\zS = \zs_6^+(\zs_3^+\zS).\] 

\end{example}

This example raises the question which indecomposable modules over a cluster-tilted algebra do not lie on a local slice. We   answer this question in a forthcoming publication \cite{ASS2}.


\section{Tubes}

The objective of this section is to show how to construct those tubes of a tame cluster-tilted algebra which contain projectives.  Let $B$ be a cluster-tilted algebra of euclidean type, and let $\mathcal{T}$ be a tube in $\Gamma(\text{mod}\,B)$ containing at least one projective.  First, consider the transjective component of $\Gamma(\text{mod}\,B)$.  Denote by $\Sigma_L$ a local slice in the transjective component that precedes all indecomposable injective $B$-modules lying in the transjective component.  Then $B/ \text{Ann}_B \Sigma_L=C_1$ is a tilted algebra having a complete slice in the preinjective component.  Define $\Sigma_R$ to be a local slice which is a successor of all indecomposable projectives lying in the transjective component.  Then $B/\text{Ann}_B \Sigma_R=C_2$ is a tilted algebra having a complete slice in the postprojective component.  Also, $C_1$ (respectively, $C_2)$ has a tube $\mathcal{T}_1$ (respectively, $\mathcal{T}_2$) containing the indecomposable projective $C_1$-modules (respectively, injective $C_2$-modules) corresponding to the projective $B$-modules in $\mathcal{T}$ (respectively, injective $B$-modules in $\mathcal{T}$).

An indecomposable projective $P(x)$ (respectively, injective $I(x)$) $B$-module that lies in a tube, is said to be a \emph{root projective} (respectively, a \emph{root injective}) if there exists an arrow in $B$ between $x$ and $y$, where the corresponding indecomposable projective $P(y)$ lies in the transjective component of $\Gamma(\text{mod}\,B)$.  

Let $\mathcal{S}_1$ be the coray in $\mathcal{T}_1$ passing through the projective $C_1$-module that corresponds to the root projective $P_B(i)$ in $\mathcal{T}$.  Similarly,  let $\mathcal{S}_2$ be the ray in $\mathcal{T}_2$ passing through the injective that corresponds to the root injective $I_B(i)$ in $\mathcal{T}$.  

Recall that if $A$ is hereditary and $T\in\text{mod}\,A$ is a tilting module, then there exists an associated torsion pair $(\torT, \ftorT)$ in $\text{mod}\,A$, where 

$$\xymatrix@R=5pt{\torT=\{M\in\text{mod}\,A\mid \text{Ext}^1_A(T,M)=0\} \\ \ftorT = \{M\in \text{mod}\,A\mid \text{Hom}_A(T,M)=0\}.}$$

\begin{lem} \label{lem tube1}
With the above notation 
\begin{itemize}
\item [\textup{(a)}] $\mathcal{S}_1 \otimes _{C_{1}} B$ is a coray in $\mathcal{T}$ passing through $P_B(i)$.
\item [\textup{(b)}] $\textup{Hom}_{C_2}(B, \mathcal{S}_2)$ is a ray in $\mathcal{T}$ passing through $I_B(i)$.  
\end{itemize}

\end{lem}

\begin{proof}
Since $C_1$ is tilted, we have $C_1=\text{End}_A T$ where $T$ is a tilting module over a hereditary algebra $A$.  As seen in the proof of Theorem 5.1 in \cite{SS}, we have a commutative diagram   

$$\xymatrix@C=60pt{\torT\ar@{^{(}->}[d]\ar[r]^{\text{Hom}_A(T,-)}&\mathcal{Y}(T)\ar[d]^{-\otimes_{C_1} B}\\
\mathcal{C}_A\ar[r]^{\text{Hom}_{\mathcal{C}_A}(T,-)}&\text{mod}\,B\;\;}$$
where $\mathcal{Y}(T)=\{N\in\text{mod}\,C\mid \text{Tor}_1^C(N,T)=0\}$. 

Let $\mathcal{T}_A$ be the tube in $\text{mod}\,A$ corresponding to the tube $\mathcal{T}$ in $\text{mod}\,B$.   By what has been seen above, we have a commutative diagram  

$$\xymatrix@C=60pt{\mathcal{T}_A \cap \torT \ar[r]^{\text{Hom}_A(T,-)}\ar[dr]_-{\text{Hom}_{\mathcal{C}_A}(T,-)\;\;\;\;\;}&\mathcal{T}_1\ar[d]^{-\otimes_{C_1} B}\\
& \mathcal{T}_1\otimes_{C_1} B \subset \mathcal{T}\;\;.}$$

Let $\mathcal{S}$ be any coray in $\mathcal{T}_1$, so it can be lifted to a coray $\mathcal{S}_{A}$ in $\mathcal{T}_A\cap \torT$ via the functor $\text{Hom}_A(T,-)$.  If we apply $\text{Hom}_{\mathcal{C}_A}(T,-)$ to this lift, we obtain a coray in $\mathcal{T}_1\otimes_{C_1} B$.  Thus, any coray in $\mathcal{T}_1$ induces a coray in $\mathcal{T}$.   Let $\mathcal{S}_1$ be the coray passing through the root projective $P_{C_1}(i)$.  Then $\mathcal{S}_1\otimes_{C_1} B$ is the coray passing through $P_{C_1}(i)\otimes_{C_1} B = P_B(i)$.  This proves (a) and part (b) is proved dually.

However, we must still justify that the ray $\mathcal{S}_1\otimes_{C_1} B$ and the coray $\text{Hom}_{\mathcal{C}_2}(B, \mathcal{S}_2)$ actually intersect (and thus lie in the same tube of $\Gamma(\text{mod}\,B)$).  Because $P_{C_1}(i)\in\mathcal{S}_1$, we have $P_{C_1}(i)\otimes B \cong P_B(i)\in \mathcal{S}_1\otimes_{C_1} B$, and $P_B(i)$ lies in a tube $\mathcal{T}$.  It is well-known that the injective $I_B(i)$ also lies in $\mathcal{T}$.   In particular, we have the following local configuration in $\mathcal{T}$, where $R$ is an indecomposable summand of the radical of $P_B(i)$ and $J$ an indecomposable summand of the quotient  of $I_B(i)$ by its  socle.  

$$\xymatrix@!C=5pt@R=5pt{I_B(i)\ar[dr]&& \circ \ar@{-->}[dr]&& P_B(i)\\
&J\ar[dr]\ar@{-->}[ur]&&R\ar[ur]\\
&&N\ar[ur]}$$

Now $I_B(i) =\text{Hom}_{C_2} (B, I_C(i))$ is coinduced, and we have shown above that the ray containing it is also coinduced.  Because $I_C(i)\in\mathcal{S}_2$, this is the ray $\text{Hom}_{C_2}(B, \mathcal{S}_2)$.  Therefore, this ray and this coray lie in the same tube, so must intersect in a module $N$, where there exists an almost split sequence 

$$\xymatrix{0\ar[r]&J\ar[r]& N \ar[r] & R \ar[r] & 0.}$$
\end{proof}

\begin{remark}
Knowing the ray $\text{Hom}_{C_2}(B, \mathcal{S}_2)$ and the coray $\mathcal{S}_1\otimes_{C_1} B$ for every root projective $P_B(i)$ in $\mathcal{T}$, one may apply the knitting procedure to construct the whole of $\mathcal{T}$.  In this way, $\mathcal{T}$ can be determined completely.  
\end{remark}

Next we show that all modules over a tilted algebra lying on the same coray change in the same way under the induction functor.  
\begin{lem}\label{lem tube2}
Let $A$ be a hereditary algebra of euclidean type, $T$ a tilting $A$-module without preinjective summands and let $C=\textup{End}_A T$ be the corresponding tilted algebra.  Let $\mathcal{T}_A$ be a tube in $\textup{mod}\,A$ and $T_i\in\mathcal{T}_A$ an indecomposable summand of $T$, such that $\textup{pd}\, I_C(i)=2$.  

Then there exists an $A$-module $M$ on the mouth of $\mathcal{T}_A$ such that we have 
$$\tau_C\Omega_C I_C(i) = \textup{Hom}_A (T,M)$$
in $\textup{mod}\,C$.  In particular, the module  $\tau_C\Omega_C I_C(i)$ lies on the mouth of the tube $\textup{Hom}_A(T, \mathcal{T}_A\cap \torT)$ in $\textup{mod}\,C$.  
\end{lem}

\begin{proof}
The injective $C$-module $I_C(i)$ is given by 
$$I_C(i) \cong \text{Ext}_A^1(T, \tau T_i) \cong D\text{Hom}_A(T_i, T),$$
where the first identity holds by \cite[Proposition VI 5.8]{ASS} and the second identity is the Auslander-Reiten formula.  Moreover, since $T_i$ lies in the tube $\mathcal{T}_A$ and $T$ has no preinjective summands, we have $\text{Hom}(T_i, T_j) \not=0$ only if $T_j$ lies in the hammock starting at $T_i$.  Furthermore, if $T_j$ is a summand of $T$ then it must lie on a sectional path starting from $T_i$ because $\text{Ext}^1(T_j, T_i)=0$.  This shows that a point $j$ is in the support of $I_C(i)$ if and only if there is a sectional path $T_i\to \dots \to T_j$ in $\mathcal{T}_A$.   We shall distinguish two cases. 

\medskip
Case 1.  If $T_i$ lies on the mouth of $\mathcal{T}_A$ then let $\omega$ be the ray starting at $T_i$ and denote by $T_1$ the last summand of $T$ on this ray.  Let $L_1$ be the direct predecessor of $T_1$ not on the ray $\omega$.  Thus we have the following local configuration in $\mathcal{T}_A$.  

$$\xymatrix@!C=5pt@!R=5pt{\ar[dr]&&\tau T_i \ar[dr] && T_i \ar[dr]\\
&\ar[ur]\ar[dr]&&\ar[ur]\ar[dr]&&\ar@{..}[dr]\\
&&\ar[ur]\ar@{..}[dr]&&\ar@{..}[dr]&&\ar[dr]\\
&&&\ar[dr]&&\tau T_1\ar[dr]\ar[ur]&&T_1\ar[dr]&\\
&&&&\tau L_1\ar[ur]\ar[dr]&&L_1\ar[ur]&&\tau^{-1} L_1 \\
&&&&&E_1\ar[ur]}$$

Then $I_C(i)$ is uniserial with simple top $S(1)$.  Moreover there is a short exact sequence 

$$\xymatrix{0\ar[r]&\tau T_i\ar[r]& L_1\ar[r] & T_1\ar[r] & 0}$$
and applying $\text{Hom}_A(T, -)$ yields 

\begin{equation}\label{s1}
\xymatrix{0\ar[r]&\text{Hom}_A(T, L_1)\ar[r]& P_C(1) \ar[r]^{f} & I_C(i) \ar[r] & \text{Ext}^1(T, L_1)\ar[r]&0}
\end{equation}

By the Auslander-Reiten formula, we have $\text{Ext}^1(T, L_1) \cong D\text{Hom}(\tau^{-1}L_1, T)$ and this is zero because $T_1$ is the last summand of $T$ on the ray $\omega$.  Thus the sequence (\ref{s1}) is short exact, the morphism $f$ is a projective cover, because $I_C(i)$ is uniserial, and hence 

$$\Omega_C I_C(i) \cong \text{Hom}_A(T, L_1).$$
Applying $\tau_C$ yields 

$$\tau_C \Omega_C I_C(i) \cong \tau_C \text{Hom}_A(T, L_1).$$

Let $E_1$ be the indecomposable direct predecessor of $L_1$ such that the almost split sequence ending at $L_1$ is of the form 

\begin{equation}\label{s2}
\xymatrix{0\ar[r]&\tau L_1 \ar[r] & E_1\oplus \tau T_1 \ar[r] & L_1 \ar[r] &0}
\end{equation}

We claim that $E_1\in \torT$.  

Recall that $L_1$ is not a summand of $T$ because $\Omega_C I_C (i) = \text{Hom}_A(T,L_1)$ is non projective.  Also, recall that $T_1$ is the last summand of $T$ on the ray $\omega$.  Suppose $E_1\not\in\torT$, thus $0\not=\text{Ext}^1_A(T, E_1) = D\text{Hom} (\tau^{-1} E_1, T)$.  Then it follows that there is a summand of $T$ on the ray $\tau\omega$ that is a successor of $\tau^{-1}E_1$.  Let $T^1$ denote the first such indecomposable summand.

$$\xymatrix@!C=5pt@!R=5pt{\ar[dr]&&\tau T_1\ar[dr] && T_1\ar[dr]\\
&\tau L_1\ar[ur]\ar[dr] && L_1 \ar[dr]\ar[ur]&&\tau^{-1}L_1\ar[dr]\\
&&E_1 \ar[ur] && \tau^{-1} E_1\ar[dr] \ar[ur]&& \ar@{..}[dr]\\
&&&&&\ar@{..}[dr]&& N \ar[dr]\\
&&&&&&T^1\ar[ur]\ar[dr] && \ar@{..}[dr]\\
&&&&&&&\ar@{..}[dr]&& \omega  \\
&&&&&&&&\tau\omega &&&\\
}$$

Then we have a short exact sequence 

$$\xymatrix{0\ar[r]& L_1\ar[r]^-{h}&T_1\oplus T^1 \ar[r] & N \ar[r] & 0}$$ 
with $h$ an $\text{add}\,T$-approximation.  Applying $\text{Hom}_A(-, T)$ yields 

$$\xymatrix@R=5pt{0\ar[r]&\text{Hom}_A(N,T)\ar[r]& \text{Hom}_A(T_1\oplus T^1, T)\ar[r]^-{h^*}& \text{Hom}_A(L_1, T)\\
&\hspace{2cm}\ar[r]&\text{Ext}^1_A(N,T)\ar[r]&0\hspace{1.5cm}}$$
and since $h$ is an $\text{add}\,T$-approximation, the morphism $h^*$ is surjective.  Thus $\text{Ext}^1_A(N,T) = 0$.  

On the other hand, $T_1\oplus T^1$ generates $N$, so $N\in\text{Gen}\,T = \torT$, and thus $\text{Ext}^1_A(T,N)=0$.  But then both $\text{Ext}^1_A(T,N) = \text{Ext}^1_A(N,T)=0$ and we see that $N$ is a summand of $T$.  This is a contradiction to the assumption that $T_1$ is the last summand of $T$ on the ray $\omega$.  Thus $E_1\in\torT$.  

Therefore, in the almost split sequence (\ref{s2}), we have $L_1, E_1 \in \torT$ and $\tau T_1 \in \ftorT$.  Moreover, all predecessors of $\tau T_1$ on the ray $\tau \omega$ are also in $\ftorT$ because the morphisms on the ray are injective.   Since $\text{Hom}_A(T, -):\; \torT \to \mathcal{Y}(T)$ is an equivalence of categories, it follows that $\text{Hom}_A(T, L_1)$ has only one direct predecessor  
$$\text{Hom}_A(T, E_1)\to \text{Hom}_A(T, L_1)$$
in $\text{mod}\,C$ and this irreducible morphism is surjective.   The kernel of this morphism is $\text{Hom}_A(T, t(\tau_A L_1))$ where $t$ is the torsion radical.  Thus we get 
$$\tau_C\Omega_C I_C(i) = \tau_C \text{Hom}_A(T, L_1) = \text{Hom}_A(T, t(\tau_A L_1)).$$

We will show that $t(\tau_A L_1)$ lies on the mouth of $\mathcal{T}_A$ and this will complete the proof in case 1.  

Let $M$ be the indecomposable $A$-module on the mouth of $\mathcal{T}_A$ such that the ray starting at $M$ passes through $\tau_A L_1$.  Thus $M$ is the starting point of the ray $\tau^{2} \omega$.  Then there is a short exact sequence of the form 

\begin{equation}\label{s3}
\xymatrix{0\ar[r]& M \ar[r] & \tau_A L_1 \ar[r] & \tau_A T_1 \ar[r] & 0}
\end{equation}
with $\tau_A T_1\in \ftorT$.  We claim that $M\in \torT$.

Suppose to the contrary that $0\not=\text{Ext}^1_A(T,M)=D\text{Hom}_A(\tau^{-1}M, T)$.  Since $\tau^{-1}M$ lies on the mouth of $\mathcal{T}_A$, this implies that there is a direct summand $T^1$ of $T$ which lies on the ray $\tau \omega$ starting at $\tau^{-1}M$.  Since $T$ is tilting, $T^1$ cannot be a predecessor of $\tau T_1$ on this ray and since $L_1$ is not a summand of $T$, we also have $L_1\not=T^1$.  Thus $T^1$ is a successor of $L_1$ on the ray $\tau\omega$.  This is impossible since such a $T^1$ would satisfy $\text{Ext}^1_A(T^1, E_1)\not=0$ contradicting the fact that $E_1\in \torT$.  

Therefore, $M\in\torT$ and the sequence (\ref{s3}) is the canonical sequence for $\tau_A L_1$ in the torsion pair $(\torT, \ftorT)$.  This shows that $t(\tau_A L_1) = M$ and hence $\tau_C\Omega_C I_C(i) = \text{Hom}_A(T,M)$ as desired.  

\medskip
 Case 2.  Now suppose that $T_i$ does not lie on the mouth of $\mathcal{T}_A$.  Let $\omega_1$ denote the ray passing through $T_i$ and $\omega_2$ the coray passing through $T_i$.   Denote by $T_1$ the last summand of $T$ on $\omega_1$, by $T_2$ the last summand of $T$ on $\omega_2$, and by $L_j$ the direct predecessor of $T_j$ which does not lie on $\omega_j$.   Note that $L_2$ does not exist if $T_2$ lies on the mouth of $\mathcal{T}_A$, and in this case we let $L_2=0$.  Thus we have the following local configuration in $\mathcal{T}_A$.  
 
$$\xymatrix@!C=5pt@!R=5pt{M\ar[dr] &&&& && &&&  & L_2 \ar[dr] && \tau^{-1}L_2\\
&\ar@{..}[dr]&&&&&&&&\tau T_2\ar[ur]\ar[dr]\ar@{..}[dl]&& T_2\ar[ur]&&\\
&&\ar@{..}[dr]&&&&&&&&\ar[ur]\ar@{..}[dl]\\
&&&\ar@{..}[dr]&&\ar[dr]\ar@{..}[ur]&&\ar[dr]\ar@{..}[ur]&&&&&&&&\\
&&&&\tau^2 T_i\ar[ur]\ar[dr] && \tau T_i \ar[dr]\ar[ur] && T_i \ar[ur]\ar[dr]\\
&&&&&\ar[ur]\ar@{..}[dr]&&\ar[ur]\ar@{..}[dr]&&\ar@{..}[dr]&&\\
&&&&&&\ar[dr]&&\tau T_1\ar[dr] && T_1\ar[dr] \\
&&&&&&&\tau L_1 \ar[ur] && L_1\ar[ur] && \tau^{-1}L_1}$$ 

The injective $C$-module $I_C(i) = \text{Ext}^1_A(T, \tau T_i)$ is biserial with top $S(1)\oplus S(2)$.  Moreover, there is a short exact sequence 

$$\xymatrix{0\ar[r]&\tau T_i \ar[r] & L_1\oplus L_2 \oplus T_i \ar[r] & T_1\oplus T_2 \ar[r] & 0.}$$
Applying $\text{Hom}_A(T,-)$ yields the following exact sequence.  

\begin{equation}\label{s4}
\xymatrix@R=5pt{0\ar[r]&\text{Hom}_A(T, L_1\oplus L_2)\oplus P_C(i) \ar[r] & P_C(1)\oplus P_C(2) \ar[r]^-{f}&I_C(i)\\
&\hspace{3cm}\ar[r]&\text{Ext}^1_A(T, L_1\oplus L_2)\ar[r]&0.}
\end{equation}

By the same argument as in case 1, using that $T_1$ and $T_2$ are the last summands of $T$ on $\omega_1$ and $\omega_2$ respectively, we see that $\text{Ext}^1_A(T, L_1\oplus L_2)=0$.  Therefore, the sequence (\ref{s4}) is short exact.  Moreover, the morphism $f$ is a projective cover and thus 
$$\Omega_C I_C(i) = \text{Hom}_A(T, L_1\oplus L_2) \oplus P_C (i).$$  
Applying $\tau_C$ yields 
$$\tau_C\Omega_C I_C(i) = \tau_C \text{Hom}_A(T, L_1) \oplus \tau_C \text{Hom}_A(T, L_2).$$

By the same argument as in case 1 we see that 

$$\tau_C \text{Hom}_A (T, L_1) = \text{Hom}_A (T, t(\tau_A L_1)) = \text{Hom}_A(T, M)$$
where $M$ is the indecomposable $A$-module on the mouth of $\mathcal{T}_A$ such that the ray starting at $M$ passes through $\tau L_1$.  In other words, $M$ is the starting point of the ray $\tau^2 \omega$.  

Therefore, it only remains to show that $\tau_C \text{Hom}_A(T, L_2) = 0$.  To do so, it suffices to show that $L_2$ is a summand of $T$.  

We have already seen that $\text{Ext}^1_A(T, L_2)=0$.  We show now  that we also have $\text{Ext}^1_A(L_2, T)=0$.  Suppose the contrary.  Then there exists a non-zero morphism $u:\, T\to \tau_A L_2$. Composing it with the irreducible injective morphism $\tau_A L_2\to \tau_A T_2$ yields a non-zero morphism in $\text{Hom}_A(T, \tau_A T_2)$.  But this is impossible since $T$ is tilting.  

Thus we have $\text{Ext}^1_A(T, L_2)= \text{Ext}^1_A(L_2, T)=0$ and thus $L_2$ is a summand of $T$, the module $\text{Hom}_A(T, L_2)$ is projective and $\tau_C \text{Hom}_A (T, L_2)=0$.  This completes the proof.  
\end{proof}

\begin{remark}
The module $M$ in the statement of the lemma is the starting point of the ray passing through $\tau^2 T_i$.  
\end{remark}

\begin{cor}\label{cor5}
Let $A, T, C, \mathcal{T}_A$ be as in Lemma \ref{lem tube2}, and let $B = C \ltimes E$, with $E = \textup{Ext}^2_C(DC,C)$.  Let $X,Y$ be two modules lying on the same coray in the tube $\textup{Hom}_A(T, \mathcal{T}_A\cap \torT)$ in $\textup{mod}\,C$.  Then $X\otimes_C E \cong Y\otimes _C E$ and thus the two projections $X\otimes_C B \to X \to 0$ and $Y\otimes _C B \to Y \to 0$ have isomorphic kernels. 
\end{cor}


\begin{proof}
For all $C$-modules $X$ we have 
$$X\otimes _B E \cong D\text{Hom} (X, DE) \cong D\text{Hom}(X, \tau_C\Omega_C DC)$$
where the first isomorphism is \cite[Proposition 3.3]{SS} and the second is \cite[Proposition 4.1]{SS}.  Since $T$ has no preinjective summands, and $X$ is regular, the only summand of $\tau\Omega DC$ for which $\text{Hom}(X, \tau\Omega DC)$ can be nonzero, must lie in the same tube as $X$.  By the lemma, the only summands of $\tau\Omega DC$ in the tube lie on the mouth of the tube.  Let $M$ denote an indecomposable $C$-module on the mouth of a tube.  Then 

$$\text{Hom}_C(X,M) \cong \text{Hom}_C (Y,M) \cong \left \{ \begin{array}{ll}k & \text{if }\, M \, \text{ lies on the coray passing}\\
&\text{ through } \, X \text{ and }\, Y, \\
0 & \text{otherwise.}\end{array}\right.$$
\end{proof}

We summarize the results of this section in the following proposition.
\begin{prop}\label{prop5}
 \begin{itemize}
\item [\textup{(a)}] Let $\cals_1$ be the coray in $\zG(\textup{mod}\,C_1)$ passing through the projective $C_1$-module corresponding to the root projective $P_B(i)$ Then $\cals_1\otimes_{C_1} B$ is a coray in $\zG(\textup{mod}\,B)$ passing through $P_B(i)$. Furthermore all modules in $\cals_1\otimes_{C_1} B $ are extensions of modules of $\cals_1$ by the same module $P_{C_1}(i)\otimes E$.
\item [\textup{(b)}] Let $\cals_2$ be the ray in $\zG(\textup{mod}\,C_2)$ passing through the injective $C_2$-module corresponding to the root injective $I_B(i)$ Then $\Hom_{C_2}(B,\cals_2)$ is a ray in $\zG(\textup{mod}\,B)$ passing through $I_B(i)$. Furthermore all modules in $\Hom_{C_2}(B,\cals_2)$ are extensions of modules of $\cals_2$ by the same module $\Hom_{C_2}(E,I_{C_2}(i))$.

\end{itemize}
\end{prop}
\begin{proof}
 (a) The first statement is Lemma \ref{lem tube1}, and the second statement is a restatement of  Corollary \ref{cor5}. 
\end{proof}
\begin{example}
Let $B$ be the cluster-tilted algebra given by the quiver 

$$\xymatrix@C=15pt@R=15pt{1\ar@<2pt>[rr]^{\lambda}\ar@<-2pt>[rr]_{\beta}&&5\ar[dl]^{\epsilon}\\
&3\ar[ul]^{\alpha}\ar[dr]^{\delta}\\
2\ar[ur]^{\gamma}&&4\ar[ll]^{\sigma}}$$
bound by $\alpha\beta=0, \beta\epsilon=0, \epsilon\alpha=0, \gamma\delta=0, \sigma\gamma=0, \delta\sigma=0$.  The algebras $C_1$ and $C_2$ are respectively given by the quivers 

$$\xymatrix@C=15pt@R=15pt{1\ar@<2pt>[rr]^{\lambda}\ar@<-2pt>[rr]_{\beta}&&5&&&& 1\ar@<2pt>[rr]^{\lambda}\ar@<-2pt>[rr]_{\beta}&&5\ar[dl]^{\epsilon}  \\
&3\ar[dr]^{\delta} \ar[ul]^{\alpha}&&&\text{and} & && 3\ar[dr]^{\delta}\\
2\ar[ur]^{\gamma}&&4&&&& 2\ar[ur]^{\gamma}&&4}$$
with the inherited relations.  We can see the tube in $\Gamma(\textup{mod}\, C_{1})$ below and the coray passing through the root projective $P_{C_1}(3) = \begin{smallmatrix}3\\4\;1\\\;\;\;5\end{smallmatrix}$ is given by 

$$  \xymatrix{\mathcal{S}_1: &\dots\ar[r]&  {\begin{smallmatrix}1\\5\end{smallmatrix}}\ar[r]&{\begin{smallmatrix}3\\4\;1\\\;\;\;5\end{smallmatrix}}\ar[r] &{\begin{smallmatrix}3\\1\\5\end{smallmatrix}}\ar[r]&{\begin{smallmatrix}2\\3\\1\\5\end{smallmatrix}}.}$$

$$\xymatrix@!C=15pt@!R=15pt{&&&&&& {\begin{smallmatrix}4\end{smallmatrix}}\ar[dr]\ar@{--}[dd]\\
&&&&&&&{\begin{smallmatrix}3\\4\;1\\\;\;\;5\end{smallmatrix}}\\
&&&& {\begin{smallmatrix}2\\3\\1\\5\end{smallmatrix}}\ar[dr]&&{\begin{smallmatrix}1\\5\end{smallmatrix}}\ar[ur]\ar@{--}[d]\\
&{\begin{smallmatrix}4\end{smallmatrix}}\ar[dr]\ar@{--}[dd]&&{\begin{smallmatrix}3\\1\\5\end{smallmatrix}}\ar[ur]\ar[dr]&&{\begin{smallmatrix}2\;\;\\3\;\;\\1\;1\\\;5\;5\end{smallmatrix}}\ar[ur]&&\\
&&{\begin{smallmatrix}3\\4\;1\\\;\;\;5\end{smallmatrix}}\ar[ur]\ar[dr]&&{\begin{smallmatrix}3\\1\;1\\\;5\;5\end{smallmatrix}}\ar[ur]&&\\
&{\begin{smallmatrix}1\\5\end{smallmatrix}}\ar[ur]\ar[dr]\ar@{--}[d]&&{\begin{smallmatrix}3\\4\;1\;1\\\;\;\;\;5\;5\end{smallmatrix}}\ar[ur]\\
\ar[ur]&&{\begin{smallmatrix}1\;1\\\;5\;5\end{smallmatrix}}\ar[ur]}$$

Dually, the ray in $\Gamma(\textup{mod}\, C_2)$ passing through the root injective $I_{C_2}(3) = {\begin{smallmatrix}1\;\;\;\\5\;2\\3\end{smallmatrix}}$ is given by 

$$\xymatrix{\mathcal{S}_2: & {\begin{smallmatrix}1\\5\\3\\4\end{smallmatrix}}\ar[r]&{\begin{smallmatrix}1\\5\\3\end{smallmatrix}}\ar[r]&{\begin{smallmatrix}1\;\;\;\\5\;2\\3\end{smallmatrix}}\ar[r]&{\begin{smallmatrix}1\\5\end{smallmatrix}}\ar[r]&\dots }$$

The root projective $P_B(3)$ lies on the coray 

$$\xymatrix{\mathcal{S}_1\otimes_{C_1} B: & \dots\ar[r]&  {\begin{smallmatrix}1\\5\\3\\4\end{smallmatrix}}\ar[r]&{\begin{smallmatrix}3\\4\;1\\\;\;\;5\\\;\;\;3\\\;\;\;4\end{smallmatrix}}\ar[r] &{\begin{smallmatrix}3\\1\\5\\3\\4\end{smallmatrix}}\ar[r]&{\begin{smallmatrix}2\\3\\1\\5\\3\\4\end{smallmatrix}}}$$

and the root injective $I_B(3)$ lies on the ray 

$$\xymatrix{\textup{Hom}_{C_2}(B, \mathcal{S}_2): & {\begin{smallmatrix}2\\3\\1\\5\\3\\4\end{smallmatrix}}\ar[r]&{\begin{smallmatrix}2\\3\\1\\5\\3\end{smallmatrix}}\ar[r]&{\begin{smallmatrix}2\;\;\;\\3\;\;\;\\1\;\;\;\\5\;2\\3\end{smallmatrix}}\ar[r]&{\begin{smallmatrix}2\\3\\1\\5\end{smallmatrix}}\ar[r]&\dots }$$

Note that by Proposition~\ref{prop5}, every module in $\mathcal{S}_1\otimes _{C_1} B$ is an extension of a module in $\mathcal{S}_1$ by $\begin{smallmatrix}3\\4\end{smallmatrix}$.  Similarly, every module in $\textup{Hom}_{C_2}(B, \mathcal{S}_2)$ is an extension of a module in $\mathcal{S}_2$ by  $\begin{smallmatrix}2\\3\end{smallmatrix}$. 

Applying the knitting algorithm we obtain the tube in $\Gamma(\textup{mod}\,B)$ containing both $\mathcal{S}_1\otimes_{C_1} B$ and $\textup{Hom}_{C_2}(B, \mathcal{S}_2)$.  

$$\xymatrix@!C=15pt@!R=15pt{
{\begin{smallmatrix}4\\2\end{smallmatrix}}\ar[dr]\ar@{--}[dd]&&\circ&&{\begin{smallmatrix}2\\3\\1\\5\\3\\4\end{smallmatrix}}\ar[dr]&&\circ&&{\begin{smallmatrix}4\\2\end{smallmatrix}}\ar@{--}[dd]\\
&{\begin{smallmatrix}4\end{smallmatrix}}\ar[dr]&&{\begin{smallmatrix}3\\1\\5\\3\\4\end{smallmatrix}}\ar[dr]\ar[ur]&&{\begin{smallmatrix}2\\3\\1\\5\\3\end{smallmatrix}}\ar[dr]&&{\begin{smallmatrix}2\end{smallmatrix}}\ar[ur]\\
\circ \ar@{--}[dd]&&{\begin{smallmatrix}3\\4\;1\\\;\;\;5\\\;\;\;3\\\;\;\;4\end{smallmatrix}}\ar[ur]\ar[dr]&&{\begin{smallmatrix}3\\1\\5\\3\end{smallmatrix}}\ar[ur]\ar[dr]&&{\begin{smallmatrix}2\;\;\;\\3\;\;\;\\1\;\;\;\\5\;2\\3\end{smallmatrix}}\ar[ur]\ar[dr]&&\circ\ar@{--}[dd]\\
&{\begin{smallmatrix}1\\5\\3\\4\end{smallmatrix}}\ar[ur]\ar[dr]&&{\begin{smallmatrix}3\\4\;1\\\;\;\;\;5\\\;\;\;3\end{smallmatrix}}\ar[ur]\ar[dr]&&{\begin{smallmatrix}3\;\;\;\\1\;\;\;\\5\;2\\3\end{smallmatrix}}\ar[ur]\ar[dr]&&{\begin{smallmatrix}2\\3\\1\\5\end{smallmatrix}}\ar[dr]\\
{\begin{smallmatrix}2\;\;\;\\3\;\;\;\\\;1\;1\\\;\;5\;5\\\;\;\;\;\;3\\\;\;\;\;\;4\end{smallmatrix}}\ar[ur]\ar[dr]\ar@{--}[d]&&{\begin{smallmatrix}1\\5\\3\end{smallmatrix}}\ar[ur]\ar[dr]&&{\begin{smallmatrix}3\;\;\;\\4\;1\\\;\;\;5\;2\\\;\;\;\;\;\;3\end{smallmatrix}}\ar[ur]\ar[dr]&&{\begin{smallmatrix}3\\1\\5\end{smallmatrix}}\ar[ur]\ar[dr]&&{\begin{smallmatrix}2\;\;\;\\3\;\;\;\\\;1\;1\\\;\;5\;5\\\;\;\;\;\;3\\\;\;\;\;\;4\end{smallmatrix}}\ar@{--}[d]\\
&\ar[ur]&&\ar[ur]&&\ar[ur]&&\ar[ur]&}$$

\end{example}

\section{From cluster-tilted algebras to quasi-tilted algebras}

Let $B$ be cluster-tilted of euclidean type $Q$ and let $A=kQ$. Then there exists $T\in \calc_A$ tilting such that $B=\End_{\calc_A}T$.

Because $Q$ is euclidean, $\calc_A$ contains at most 3 exceptional tubes.
Denote by $T_{0},T_1,T_2,T_3$ the direct sums of those summands of $T$ that respectively lie in the transjective component and in the three exceptional tubes.

In the derived category $\DA$, we can choose a lift of $T$ such that we have the following local configuration.
\smallskip

\begin{center}
\scalebox{.65}{ 
\begingroup%
  \makeatletter%
  \providecommand\color[2][]{%
    \errmessage{(Inkscape) Color is used for the text in Inkscape, but the package 'color.sty' is not loaded}%
    \renewcommand\color[2][]{}%
  }%
  \providecommand\transparent[1]{%
    \errmessage{(Inkscape) Transparency is used (non-zero) for the text in Inkscape, but the package 'transparent.sty' is not loaded}%
    \renewcommand\transparent[1]{}%
  }%
  \providecommand\rotatebox[2]{#2}%
  \ifx\svgwidth\undefined%
    \setlength{\unitlength}{519.40868029bp}%
    \ifx\svgscale\undefined%
      \relax%
    \else%
      \setlength{\unitlength}{\unitlength * \real{\svgscale}}%
    \fi%
  \else%
    \setlength{\unitlength}{\svgwidth}%
  \fi%
  \global\let\svgwidth\undefined%
  \global\let\svgscale\undefined%
  \makeatother%
  \begin{picture}(1,0.2164963)%
    \put(0,0){\includegraphics[width=\unitlength]{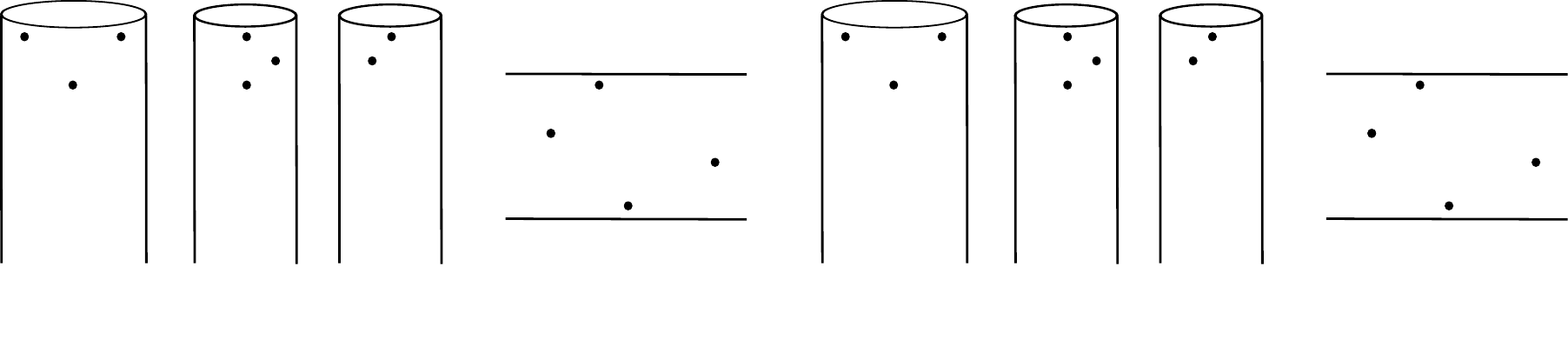}}%
    \put(0.04261005,0.00440254){\color[rgb]{0,0,0}\makebox(0,0)[lb]{\smash{$T_1$}}}%
    \put(0.14734454,0.00440254){\color[rgb]{0,0,0}\makebox(0,0)[lb]{\smash{$T_2$}}}%
    \put(0.23667689,0.00440254){\color[rgb]{0,0,0}\makebox(0,0)[lb]{\smash{$T_3$}}}%
    \put(0.38145691,0.00440254){\color[rgb]{0,0,0}\makebox(0,0)[lb]{\smash{$T_{0}$}}}%
    \put(0.5570412,0.00440254){\color[rgb]{0,0,0}\makebox(0,0)[lb]{\smash{$FT_1$}}}%
    \put(0.66485611,0.00440254){\color[rgb]{0,0,0}\makebox(0,0)[lb]{\smash{$FT_2$}}}%
    \put(0.75726889,0.00440254){\color[rgb]{0,0,0}\makebox(0,0)[lb]{\smash{$FT_3$}}}%
    \put(0.90512933,0.00440254){\color[rgb]{0,0,0}\makebox(0,0)[lb]{\smash{$FT_{0}$}}}%
  \end{picture}%
\endgroup%
 }
\end{center}

Let $\calh$ be a hereditary category that is derived equivalent to $\textup{mod}\,A$ and such that $\calh$ is not the module category of a hereditary algebra. Then $\calh$ is of the form $\calh=\calt^-\vee\calc\vee\calt^+$, where $\calt^-, \calt^+$ consist of tubes, and $\calc$ is a transjective component, see \cite{LenzingSkowronski}.
Let $T_-$, $T_+$  be the direct sum of all indecomposable summands of $T$ lying in $\calt^-$,  $\calt^+ $ respectively.
We define two subspaces $L$ and $R$ of $B$ as follows.
\[L=\Hom_{\DA}(F^{-1}T_+,T_{0}) \quad \textup{and} \quad R=\Hom_{\DA}(T_{0},FT_-).\]

The transjective component of $\textup{mod}\,B$ contains a left section $\zS_L$ and a right section $\zS_R$, see \cite{Assem}. Thus $\zS_L,\zS_R$ are local slices,  $\zS_L$  has no projective predecessors, and $\zS_R $ has no projective successors in the transjective component.
Define $K$ to be the two-sided ideal of $B$ generated by $ \Ann\, \zS_L\cap\Ann\, \zS_R $ and the two subspaces $L$ and $R$. Thus
\[K=\langle \Ann\, \zS_L\cap\Ann\, \zS_R , L, R\rangle.\]
We call $K$ the {\em partition ideal} induced by the partition $\calt^-\vee \calc\vee\calt^+$.
\begin{thm}
 \label{thm ctaqt}
 The algebra $C=B/K$ is quasi-tilted and such that $B=\Ctilde$. Moreover $C$ is tilted if and only if $L=0$ or $R=0$.
\end{thm}
\begin{proof}
 We have
 $B=\End_{\calc_A} T =\oplus_{i\in\ZZ}\Hom_{\DA}(T,F^iT)$, where the last equality is as $k$-vector spaces. 
 Using the decomposition $T=T_-\oplus T_0 \oplus T_+$, we see that $B$ is equal to
\[ 
\begin{array}
 {ccccccccccc}
 & \Hom_{\cald}(T_-,T_-) &\oplus& \Hom_{\cald}(T_-,T_0) &\oplus&\Hom_{\cald}(T_-,FT_-) \\
  \oplus& \Hom_{\cald}(T_0,T_0) &\oplus& \Hom_{\cald}(T_0,T_+)&\oplus& \Hom_{\cald}(T_0,FT_-) \\
   \oplus& \Hom_{\cald}(T_0,FT_0)
     &\oplus& \Hom_{\cald}(F^{-1}T_+,FT_0)&\oplus& \Hom_{\cald}(F^{-1}T_+,T_+)
\\
\oplus& \Hom_{\cald}(T_+,T_+),
\end{array}
\] 
where all Hom spaces are taken in $\DA$. On the other hand,
\[ 
\begin{array}
 {ccccccccccc}
 \End_{\calh} T&=& \Hom_{\calh}(T_-,T_-) &\oplus& \Hom_{\calh}(T_-,T_0) &\oplus&\Hom_{\calh} (T_0,T_0)\\ &\oplus& \Hom_{\calh}(T_0,T_+)&\oplus& \Hom_{\calh}(T_+,T_+)
\end{array}
\] 
is a quasi-tilted algebra. Thus in order to prove that $C$ is quasi-tilted it suffices to show that $K$ is the ideal generated by 
\[\Hom_{\cald}(T_-,FT_-) \oplus \Hom_{\cald}(T_0,FT_-\oplus FT_0) 
\oplus \Hom_{\cald}(F^{-1}T_+,T_0\oplus T_+). 
\]
But this follows from the definition of $L$ and $R$ and the fact that the annihilators of the local slices $\zS_L$ and $\zS_R$ are given by the morphisms in $\End_{\calc_A}T$ that factor through the lifts of the corresponding local slice in the cluster category. More precisely,
\[\begin{array}{rcl}\Ann \,\zS_L &\cong&\Hom_{\cald}(F^{-1} T_0\oplus F^{-1}T_+\oplus T_- \ ,\  T_0\oplus T_+ \oplus FT_-),\\
\Ann \,\zS_R &\cong&\Hom_{\cald}( F^{-1}T_+\oplus T_-\oplus T_0 \ ,\  T_+ \oplus FT_-\oplus FT_0),
\end{array}\]
and thus
\[\begin{array}{rcl}\Ann \,\zS_L\cap\Ann\,\zS_R &\cong&\Hom_{\cald}( T_0,F  T_0)\oplus \Hom_{\cald}( T_-,F  T_-)\\&&\ \oplus \Hom_{\cald}( F^{-1}T_+,T_+), 
\end{array}\]
where we used the fact that $\Hom_{\cald}(T_-,T_+)=\Hom_{\cald}(T_+,T_-)=0$. This completes the proof that $C$ is quasi-tilted.

Since $C=\End_\calh T$, we have $\Ctilde=\End_{\calc_\calh}T\cong\End_{\calc_A}T=B.$

Now assume that $R=0$. Then $T_-=0$ and thus $K$ is generated by
$(\Ann\,\zS_L\cap\Ann\,\zS_R)\oplus L$, and this is equal to 
\begin{equation}\label{eq ctaqt} \Hom_\cald(T_0,FT_0)\oplus \Hom_\cald(F^{-1}T_+,T_+)\oplus \Hom_\cald(F^{-1}T_+,FT_0).
\end{equation}
On the other hand, $T_-=0$ implies that 
\[\Ann\,\zS_L = \Hom_\cald(F^{-1}T_0\oplus F^{-1}T_+,T_0\oplus T_+),\]
and since $\Hom_\cald(F^{-1}T_0,T_+)=0$, this implies that
$K=\Ann\,\zS_L$ is the annihilator of a local slice. Therefore $C=B/K$ is tilted by \cite{ABS2}.
The case where $L=0$ is proved in a similar way. 

Conversely, assume $C$ is tilted. Then $K=\Ann\,\zS'$ for some local slice $\zS'$ in $\textup{mod}\,B$. We   show that  $K=\Ann\,\zS_L$ or $K=\Ann\,\zS_R$. Suppose to the contrary that $\zS'$ has both a predecessor and a successor in $\add\, T_0$. Then there exists an arrow $\za$ in the quiver of $B$ such that $\za\in\Hom_\cald(T_0,T_0)$ and $\za\in\Ann\,\zS'=K$. But by definition of $\zS_L,\zS_R,L $ and $R$, we see that this is impossible. 

Thus  $K=\Ann\,\zS_L$ or $K=\Ann\,\zS_R$. In the former case, we have $R=0$, by the computation (\ref{eq ctaqt}), and in the latter case, we have $L=0$.
 \end{proof}

\begin{thm}
 \label{thm ctaqt2}
 If $C$ is quasi-tilted of euclidean type and $B=\Ctilde$ then
 \[C=B/\Ann(\zS^-\oplus \zS^+),\]
 where $\zS^- $ is a right section in the postprojective component of $C$ and $\zS^+$ is a left section in the preinjective component.
\end{thm}
 
\begin{proof}
 $C$ being quasi-tilted implies that there is a hereditary category $\calh$ with a tilting object $T$ such that $C=\End_\calh T$. Moreover, $B=\End_{\calc_\calh} T$ is the corresponding cluster-tilted algebra. As before we use the decomposition $T=T_-\oplus T_0\oplus T_+$. 
 Then the algebras
 \[C^- =\End_\calh (T_-\oplus T_0)\quad\textup{and}\quad C^+ =\End_\calh ( T_0\oplus T_+)\] 
 are tilted. 
 Let $\zS^-$ and $ \zS^+$ be complete slices in $\textup{mod}\,C^-$ and $\textup{mod}\,C^+$ respectively. Note that $\zS^-$ lies in the postprojective component and $\zS^+$ lies in the preinjective component of their respective module categories.
 
 Then $C$ is a branch extension of $C^-$ by the module 
 \[M^+=\Hom_\calh(T_+,T_+)\oplus \Hom_\calh(T_0,T_+).\] Similarly
  $C$ is a branch coextension of $C^+$ by the module 
  \[M^-=\Hom_\calh(T_-,T_-)\oplus \Hom_\calh(T_-,T_0).\]
  Observe that the postprojective component of $C^-$ does not change under the branch extension, and the preinjective component of $C^+$ does not change under the branch coextension. Therefore $\zS^-$ is a right section in the postprojective component of $C$ and   $\zS^+$ is a left section in the preinjective component.
  Moreover, by construction, we have
\[\Ann_B\zS^- =M^+\oplus \E \quad\textup{and}\quad \Ann_B\zS^+ =M^-\oplus \E,\]
and therefore
\[\Ann_B(\zS^- \oplus \zS^+)=\Ann_B\zS^- \cap\Ann_B\zS^+= \E .
\]
This completes the proof.
\end{proof}

The main theorem of this section is the following.
 
\begin{thm}
 \label{thm ctaqt3} 
 Let $C$ be a quasi-tilted algebra whose relation-extension $B$ is cluster-tilted of euclidean type. Then $C$ is one of the following.
 \begin{itemize}
\item [\textup{(a)}]
$C=B/\Ann\,\zS$ for some local slice $\zS$ in $\zG(\textup{mod}\,B)$.
\item [\textup{(b)}]
$C=B/K$ for some partition ideal $K$.
\end{itemize}

\end{thm}
\begin{proof}
 Assume first that $C$ is tilted. Then, because of \cite{ABS2}, there exists a local slice $\zS$ in the transjective component of $\zG(\textup{mod}\,B)$ such that $B/\Ann\,\zS=C$. Otherwise, assume that $C$ is quasi-tilted but not tilted. Then, because of \cite{LenzingSkowronski}, there exists a hereditary category $\calh$ of the form \
 \[
\calh = \calt^-\vee\calc\vee\calt^+\]
and a tilting object $T$ in $\calh$ such that $C=\End_\calh T$. Because of Theorem~\ref{thm ctaqt} we get $C=B/K$ where $K$ is the partition ideal induced by the given partition of $\calh$.
\end{proof}

\begin{example}
 Let $B$ be the cluster-tilted algebra of type $\widetilde{\mathbb{E}}_7$ given by the quiver
 \[\xymatrix@C50pt{8\ar[dr] && 7\ar[ll]_\ze \\
 &6\ar[ru]\ar[rdd]^(0.33){{\zb_3}} \\
 &5\ar[rd]^(0.4){\zb_2\quad}\\
 1\ar[ruu]^(0.65){ \za_3}\ar[ru]^(0.6){ \za_2}\ar[rd]_(0.6){\za_1} && 2\ar@<1.5pt>[ll]\ar@<-1.5pt>[ll]\\
 &3\ar[ru]_(0.4){\zb_1}\ar[rd] \\ &&4
 }\]
 As usual let $T_i$ denote the indecomposable summand of $T$ corresponding to the vertex $i$ of the quiver. In this example $T$ has two transjective summands $T_1,T_2$, and the other summands lie in three different tubes.  $T_3, T_4$ lie in a tube $\calt_1$, $T_5 $ lies in a tube $\calt_2$ and $T_6,T_7,T_8$ lie in a tube $\calt_3$. 

Choosing a partition ideal corresponds to choosing a subset of tubes to be predecessors of the transjective component. Thus 
there are 8 different partition ideals corresponding to the 8 subsets of $\{\calt_1,\calt_2,\calt_3\}$.   
If the tube $\calt_i$ is chosen to be a predecessor of the transjective component, then the arrow $\zb_i$ is in the partition ideal. And if $\calt_i $ is not chosen to be a predecessor of the transjective component, then it is a successor and consequently the arrow $\za_i$ is in the partition ideal. The arrow $\ze $ is always in the partition ideal since it corresponds to a morphim from $T_8$ to $FT_7$ in the derived category.

Sumarizing, the 8 partition ideals $K$ are the ideals generated by the following sets of arrows.
\[ \{\za_{i}, \zb_j,\ze\mid   i\notin I , j\in I, I \subset\{1,2,3\}\}.\]

The quiver of the corresponding quasi-tilted algebra $B/ K$ is obtained by removing the generating arrows from the quiver of $B$. Exactly 2 of these 8 algebras are tilted, and these correspond to cutting $\za_1,\za_2,\za_3,\ze$, respectively 
$\zb_1,\zb_2,\zb_3,\ze$. 
\end{example}



\end{document}